\documentclass[12pt]{amsart}
\usepackage{latexsym,amsthm,amssymb,amsfonts}
\usepackage[colorlinks]{hyperref}

\usepackage{enumerate}
\usepackage[T5]{fontenc}
\newtheorem{theorem}{Theorem}
\newtheorem{lemma}[theorem]{Lemma}
\newtheorem{proposition}[theorem]{Proposition}
\newtheorem{corollary}[theorem]{Corollary}
\newtheorem{question}[theorem]{Question}

\theoremstyle{definition}

\theoremstyle{remark}
\newtheorem*{remark}{Remark}

\newcommand{\GL}{{\mathrm {GL}}}

\newcommand{\SL}{{\mathrm {SL}}}
\newcommand{\PSL}{{\mathrm {PSL}}}

\newcommand{\Aut}{{\mathrm {Aut}}}

\newcommand{\Irr}{{\mathrm {Irr}}}

\newcommand{\St}{{\mathrm {St}}}

\newcommand{\Sym}{{\mathrm {Sym}}}

\newcommand{\NN}{{\mathbb N}}

\newcommand{\bZ}{\mathbf{Z}}
\newcommand{\bO}{{\bf{O}}}

\newcommand{\bC}{{\mathbf{C}}}
\newcommand{\bN}{\mathbf{N}}
\newcommand{\bF}{\mathbf{F}}

\newcommand{\Al}{\textup{\textsf{A}}}
\newcommand{\Sy}{\textup{\textsf{S}}}

\newcommand{\cp}{{\mathrm {cp}}}
\newcommand{\nor}{\vartriangleleft}

\begin{document}

\title[On a conjecture of Gluck]
{On a conjecture of Gluck}

\author[J.\,P. Cossey]{James P. Cossey}
\address{Department of Mathematics, The University of Akron, Akron,
Ohio 44325, USA} \email{cossey@uakron.edu}

\author[Z. Halasi]{Zolt\'an Halasi} \address{Department of Algebra
and Number Theory, Institute of Mathematics, University of Debrecen,
4010, Debrecen, Pf.~12, Hungary} \email{halasi.zoltan@renyi.mta.hu}

\author[A. Mar\'{o}ti]{Attila Mar\'oti} \address{Fachbereich
Mathematik, Technische Universit\"{a}t Kaiserslautern, Postfach
3049, 67653 Kaiserslautern, Germany \and Alfr\'ed R\'enyi Institute of
  Mathematics, Re\'altanoda utca 13-15, H-1053, Budapest, Hungary}
\email{maroti@mathematik.uni-kl.de \and maroti.attila@renyi.mta.hu}

\author[H.\,N. Nguyen]{Hung Ngoc Nguyen}
\address{Department of Mathematics, The University of Akron, Akron,
Ohio 44325, USA} \email{hungnguyen@uakron.edu}

\dedicatory{Dedicated to Professor Nguy\~\ecircumflex n H.\,V.
H\uhorn ng on the occasion of his sixtieth birthday.}

\thanks{The research of the second author leading to these results has
received funding from the European Union's Seventh Framework Programme
(FP7/2007-2013) under grant agreement no. 318202, from ERC Limits of
discrete structures Grant No.\ 617747 and from OTKA K84233. The research
of the third author was supported by an Alexander
  von Humboldt Fellowship for Experienced Researchers, by the J\'anos
  Bolyai Research Scholarship of the Hungarian Academy of Sciences, by
  OTKA K84233, and by the MTA RAMKI Lend\"ulet Cryptography Research
  Group. The fourth author is partially supported by NSA Young Investigator Grant \#H98230-14-1-0293
    and a BCAS Faculty Scholarship Award from the Buchtel College of Arts and Sciences, The University of Akron.}

\subjclass[2010]{Primary 20C15, 20C30; Secondary 20D06}

\keywords{Finite groups, solvable groups, character degrees, largest
degree, Gluck's conjecture, Jordan-type theorems}

\date{\today}

\begin{abstract} Let $\bF(G)$ and $b(G)$ respectively denote the
  Fitting subgroup and the largest degree of an irreducible complex
  character of a finite group $G$. A well-known conjecture of D.~Gluck
  claims that if $G$ is solvable then $|G:\bF(G)|\leq b(G)^{2}$. We
  confirm this conjecture in the case where $|\bF(G)|$ is coprime to
  6. We also extend the problem to arbitrary finite groups and prove
  several results showing that the largest irreducible character
  degree of a finite group strongly controls the group structure.
\end{abstract}
\maketitle


\section{Introduction}

For a finite group $G$, let $\Irr(G)$ denote the set of irreducible
(complex) characters of $G$, and write \[b(G):=\max\{\chi(1)\mid
\chi\in\Irr(G)\},\] so that $b(G)$ is the largest irreducible
character degree of $G$. Also, let $\bF(G)$ denote the Fitting
subgroup of $G$. In an old paper, D.~Gluck~\cite{Gluck} proved that
$|G:\bF(G)|\leq b(G)^{13/2}$ and further conjectured that
$|G:\bF(G)|\leq b(G)^2$ for every solvable group $G$. Although
Gluck's conjecture is still open, various partial results have been
obtained by many authors. For instance, A.~Espuelas~\cite{Espuelas}
verified the conjecture for groups of odd order and this was
extended later by S.~Dolfi and E.~Jabara~\cite{Dolfi-Jabara} to
(solvable) groups with abelian Sylow 2-subgroups. Recently,
Y.~Yang~\cite{Yang} has confirmed the conjecture for solvable
$3'$-groups. The best general bound is due to A.~Moret\'{o} and
T.\,R.~Wolf in~\cite{Moreto-Wolf} where it was proved that
$|G:\bF(G)|\leq b(G)^3$ for every solvable group $G$.

We have seen that all the up-to-date partial results on Gluck's
conjecture have been obtained with an additional condition on the
order of $G$ or $G/\bF(G)$. In this paper we establish the
conjecture with a numerical restriction only on the order of
$\bF(G)$. Our first result may be considered not only as a
generalization of the aforementioned theorem of Moret\'o and Wolf
but also as an important step towards the solution of Gluck's
conjecture. (As usual, $\bF^\ast(G)$ and $\Phi(G)$ respectively
denote the generalized Fitting subgroup and the Frattini subgroup of
a finite group $G$.)

\begin{theorem}\label{theorem-extendMoretoWolf}
  Let $G$ be a finite $\pi$-solvable group where $\pi$ is the set of
  the prime divisors of $|{\bF}^{\ast}(G/\Phi(G))|$. Then $|G:\bF(G)|
  \leq {b(G)}^3$. Furthermore if $|\bF (G/\Phi(G))|$ is not divisible
  by $64$ nor $81$ then $|G:\bF(G)| \leq {b(G)}^2$.
\end{theorem}

Combining the second part of Theorem \ref{theorem-extendMoretoWolf}
with the results of Dolfi, Jabara and Yang mentioned above, we obtain

\begin{corollary}\label{corollary} Gluck's conjecture is true unless
  possibly if $|G/\bF(G)|$ is divisible by $6$ and $|\bF(G/\Phi(G))|$
  is divisible by $64$ or $81$.
\end{corollary}

\begin{proof} This follows from ~\cite[Theorem 3]{Dolfi-Jabara},
  \cite[Theorem~2.5]{Yang}, and
  Theorem~\ref{theorem-extendMoretoWolf}.
\end{proof}

We believe that the following result, which is a consequence of a
theorem of Dolfi~\cite{Dolfi}, will also be useful in a future
general attack on Gluck's conjecture.

\begin{theorem}\label{primitive}
  Let $G$ be a finite solvable group with $G/\bF(G)$ acting
  primitively on $\bF(G)/\Phi(G)$. Then Gluck's conjecture is true for
  $G$.
\end{theorem}

Note that in Theorem~\ref{primitive} the factor group $G/\bF(G)$
always acts completely reducibly and faithfully on $\bF(G)/\Phi(G)$,
whenever $G$ is solvable.  This is a theorem of W.~Gasch\"utz
\cite[III.~4.2,~4.4,~4.5]{Huppert}.

While the problem of bounding the index of the Fitting subgroup of a
finite group by its largest character degree has been studied
extensively for solvable groups, not much has been done for
arbitrary groups. We are aware of only one known result which is due
to Gluck and is in the same paper mentioned above. That is, there
exists a constant $c$ such that $|G:\bF(G)|\leq b(G)^c$ for every
finite group $G$. Another goal of this paper is to find an explicit
polynomial bound for $|G:\bF(G)|$ in terms of $b(G)$ for an
arbitrary group $G$. Surprisingly, by using a result of
R.\,M.~Guralnick and G.\,R.~Robinson~\cite{Guralnick-Robinson} on
the so-called commuting probability of finite groups, we easily
deduce that this unspecified constant $c$ can be taken to be $4$.

\begin{theorem}\label{main-theorem-b(G)^4} For every finite group $G$, we
  have $|G:\bF(G)|\leq b(G)^{4}$.
\end{theorem}

One might ask for the best possible bounding constant $c$. We have
observed that $|G|\leq b(G)^{3}$ for all non-abelian finite simple
groups $G$ (see Theorem~\ref{main theorem 2}) and the bounding
constant $3$ cannot be lowered as the simple linear groups
$\SL(2,2^f)$ show. (Recall that $|\SL(2,2^f)|=2^f(4^f-1)$ and
$b(\SL(2,2^f))=2^f+1$.) As we know of no finite group $G$ with
$|G:\bF(G)|>b(G)^3$ we put forward the following

\begin{question}\label{conjecture}
  Is it true that $|G:\bF(G)|\leq b(G)^{3}$ for every finite group
  $G$?
\end{question}

Indeed, Theorem~\ref{theorem-extendMoretoWolf} already answers
Question~\ref{conjecture} affirmatively in the case where $G$ is a
finite $\pi$-solvable group where $\pi$ is the set of the prime
divisors of $|{\bF}^{\ast}(G/\Phi(G))|$. To further support this
question, we answer it in the case where $G/\bF(G)$ has no abelian
composition factor.

\begin{theorem}\label{main theorem productofcompositionfactors} Let
  $G$ be a finite group. Then the product of the orders of the
  non-abelian composition factors of $G$ is at most $b(G)^3$.
\end{theorem}

Gluck's conjecture and Question~\ref{conjecture} suggest that it
would be interesting to find lower bounds for $b(G)$ in terms of
indices of distinguished subgroups of $G$. Let $k(G)$ denote the
number of conjugacy classes of $G$. Since $|G|/k(G)\leq b(G)^2$, we
see that an upper bound for $k(G)$ provides the corresponding lower
bound for $b(G)$ (see Section~\ref{section-other-bounds}). In fact,
the problem of finding upper bounds for $k(G)$ in terms of
distinguished subgroups of $G$ has been studied considerably in the
literature. One of the notable results, motivated by a question of
J.\,G.~Thompson on bounding $k(G)$ by the order of a so-called
\emph{nilpotent injector} of (the solvable group) $G$, is due to
Robinson~\cite{Robinson} who proved (among other results) that there
is a well determined function $f$ such that $k(G)\leq f(|\bF(G)|)$
for every solvable group $G$. We exploit Robinson's ideas (see also
\cite[Theorem~13 (ii)]{Guralnick-Robinson}) to obtain the following
theorem whose consequence can be viewed as a weak form of Gluck's
conjecture.

\begin{theorem}\label{main-theorem-Hallsubgroup}
  Let $G$ be a finite $\pi$-solvable group where $\pi$ is the set of
  the prime divisors of $|\bF^{*}(G)|$. Then $k(G) \leq |G|_{\pi}$ and
  consequently, if $H$ is a Hall $\pi$-subgroup of $G$ then $|G:H|\leq
  b(G)^2$. In particular, if $G$ is a solvable group and $H$ is a Hall
  $\pi$-subgroup of $G$ where $\pi$ is the set of the prime divisors
  of $|\bF(G)|$, then $|G:H|\leq b(G)^2$.
\end{theorem}


I.\,M.~Isaacs and D.\,S.~Passman~\cite{Isaacs-Passman} proved a
Jordan-type theorem which says that there exists a real-valued
function $f(x)$ so that every finite group $G$ has an abelian
subnormal subgroup of index at most $f(b(G))$ (and at least $b(G)$).
As a consequence of Theorem~\ref{main-theorem-b(G)^4}, we show that
$f(x)$ can be taken to be $x^8$. We combine this with a result of
M.\,W.~Liebeck and L.~Pyber~\cite[Theorem~3]{LiebeckPyber} on
bounding the class number of a finite group in terms of the order of
a solvable subgroup to obtain

\begin{theorem}\label{IsaacsPassman}
  Every finite group $G$ contains an abelian subnormal subgroup of
  index at most $b(G)^{8}$ and a solvable subgroup of index at most
  ${b(G)}^{2}$.
\end{theorem}

Since this paper was written L. Pyber informed us that in fact
Gluck~\cite{Gluck} has shown that there exists a universal constant
$K$ so that in any finite group $G$ there exists an abelian subgroup
of index at most $b(G)^{K}$. Pyber also pointed on a result of
A.~Chermak and A.~Delgado \cite[Theorem~1.41]{Isaacs} saying that if
a finite group contains an abelian subgroup of index $n$ then it
also contains a characteristic abelian subgroup of index at most
$n^2$. Using this result and the proof of
Theorem~\ref{IsaacsPassman}, we see that every finite group $G$
contains a characteristic abelian subgroup of index at most
$b(G)^{12}$. This and our Question~\ref{conjecture} motivated
Pyber~\cite{Pyber} to ask the following question. Is it true that
for a finite group $G$ there exists an abelian normal subgroup of
index at most $b(G)^3$?

The paper is organized as follows. In Section \ref{bases} we
introduce bases and state two background results. In
Section~\ref{orbits} we translate some information about bases to
`large' orbits. Theorems~\ref{theorem-extendMoretoWolf} and
\ref{primitive} are proved in Section~\ref{Section 4}. In
Section~\ref{sectionsimplegroups}, we prove that every non-abelian
finite simple group $S$ has an irreducible character that is
extendible to $\Aut(S)$ of degree at least $|S|^{1/3}$ - a crucial
result that will be used in
Section~\ref{section-product-arbitrary-groups} to prove
Theorem~\ref{main theorem productofcompositionfactors}. Finally,
Theorems~\ref{main-theorem-b(G)^4}, \ref{main-theorem-Hallsubgroup},
and~\ref{IsaacsPassman} are proved in
Section~\ref{section-other-bounds}.


\section{Bases}\label{bases}

The notion of a base is fundamental in permutation group theory and
also in computational group theory. For a finite permutation group
$H \leq \Sym(\Omega)$, a subset of the finite set $\Omega$ is called
a {\it base} for $H$ if its pointwise stabilizer in $H$ is the
identity. A base of minimal size is called a {\it minimal base}.

There are a number of results on minimal base sizes of linear
groups. One of these is the following.



\begin{theorem}[\cite{Halasi-Maroti}, Theorem~1.1]\label{c3}
  Let $V$ be a finite vector space over a field of characteristic
  $p$. If $G\leq \GL(V)$ is a $p$-solvable group which acts
  irreducibly on $V$, then the minimal base size for $G$ is at most
  $2$ unless $p = 2$ or $3$, when the minimal base size for $G$ is at
  most $3$.
\end{theorem}

Let us remark that Theorem~\ref{c3} is best possible. There are
infinitely many solvable linear groups $G \leq \GL(V)$ acting irreducibly on $V$
with $|G| > {|V|}^{2}$ for $p=2$ or $3$ (see
\cite[Theorem~1]{Palfy} or \cite[page~1097]{Wolf}).

To state the most general form of the next theorem we introduce yet
another definition.  Two bases $B_{1}$ and $B_{2}$ for $H \leq \GL(V)$
are said to be non-equivalent if there is no $h \in H$ with $B_{1}^h =
B_{2}$. The next result concerns such bases in the case where $H$ is a
solvable primitive linear group.

\begin{theorem}[Dolfi \cite{Dolfi}, Theorem 3.4]
  \label{Dolfi} Let $G$ be a solvable primitive subgroup of $\GL(V) =
  \GL(n,q)$ where $V$ is a vector space of dimension $n$ over a field
  of order $q$ and characteristic $p$. Then $G$ has at least $p$
  pairwise non-equivalent bases each of size $2$ unless one of the
  following cases holds for $G$.
  \begin{enumerate}
  \item[\textup{(1)}] $\GL(2,2)$;
  \item[\textup{(2)}] $\SL(2,3)$ or $\GL(2,3)$;
  \item[\textup{(3)}] $3^{1+2}.\SL(2,3)$ or $3^{1+2}.\GL(2,3) \leq \GL(6,2)$;
  \item[\textup{(4)}] $(Q_{8} \circ Q_{8})H \leq \GL(4,3)$ where $H$
    is isomorphic to a subgroup of index $1$, $2$, or $4$ of
    $\mathrm{O}^{+}(4,2)$.
  \end{enumerate}
\end{theorem}



\section{Large orbits}\label{orbits}

In this section we translate the results in the previous section to
results about `large' orbits. By a `large' orbit of a finite group
$G$ on a vector space $V$ (possibly of mixed characteristic) we mean
an orbit of length at least ${|G|}^{1/2}$ or ${|G|}^{1/3}$. The
existence of such an orbit is equivalent to the existence of a
vector $v \in V$ with the property that $|\bC_{G}(v)| \leq
|G|^{1/2}$ or $|\bC_{G}(v)| \leq |G|^{2/3}$. This occurs for example
if $G$ admits a base of size at most $2$ or $3$ on $V$. Our method
below is part of Gluck's strategy \cite{Gluck} to produce an
irreducible character of large degree (for this see
Section~\ref{Section 4}).

\begin{lemma}\label{lemma key}
  Let $G$ be a finite group acting completely reducibly and faithfully
  on a finite module $V$, possibly of mixed characteristic. Suppose
  that $G$ is $\pi$-solvable where $\pi$ is the set of prime divisors
  of $|V|$. Then we have the following.
  \begin{enumerate}
  \item[\textup{(1)}] There exists a vector $v$ in $V$ such that $|\bC_G(v)|
    \leq |G|^{2/3}$.
  \item[\textup{(2)}] If $|V|$ is not divisible by $64$ nor $81$ then
    there exists a vector $v$ in $V$ such that $|\bC_G(v)| \leq
    |G|^{1/2}$.
  \item[\textup{(3)}] If $V$ is a primitive $G$-module and $G$ is
    solvable, then there exists a vector $v$ in $V$ such that
    $|\bC_G(v)| \leq |G|^{1/2}$ unless $|V|=3^4$ and $G$ is a unique
    subgroup (up to conjugacy) in $GL(V)$ of order $1152$.
  \end{enumerate}
\end{lemma}

\begin{proof}
  Let us first prove parts (1) and (2) of the lemma. We use induction
  on the size of $V$. (If $|V|$ is a prime then the result is clear.)

  Suppose that $V$ is reducible. Put $\epsilon = 1/2$ if $|V|$ is not
  divisible by $64$ nor $81$ and $2/3$ otherwise.  We have $V = U
  \oplus W$ where $U$ and $W$ are non-zero $G$-modules of possibly
  mixed characteristic.  Then $G/N$ acts faithfully and completely
  reducibly on $U$, a module of size smaller than $|V|$, where $N =
  \bC_G(U)$. So by the induction hypothesis, there exists a vector $u$
  in $U$ with $|\bC_{G/N}(u)| \leq |G/N|^{\epsilon}$. Note that $N$
  acts faithfully and completely reducibly on $W$ and thus by
  induction, there exists a vector $w$ in $W$ with $|\bC_{N}(w)| \leq
  |N|^{\epsilon}$. Put $v = u+w$.  Then $\bC_{G}(v)N/N \leq
  \bC_{G/N}(u)$ and so \[|\bC_{G}(v): N \cap \bC_{G}(v)| \leq
  |G/N|^{\epsilon}.\] But \[|N \cap \bC_{G}(v)| = |\bC_{N}(v)| =
  |\bC_{N}(w)| \leq |N|^{\epsilon}.\] These yield
  \[|\bC_{G}(v)| \leq |G|^{\epsilon}.\]

  Thus we may assume that $V$ is irreducible of characteristic $p$.
  Then by Theorem~\ref{c3} there exists a base of size at most $3$
  and, in case $p \geq 5$, a base of size at most $2$. In particular
  there exists a vector $v$ in $V$ with the property that
  $|\bC_{G}(v)| \leq |G|^{2/3}$ and a vector $v$ in $V$ with the
  property that $|\bC_{G}(v)| \leq |G|^{1/2}$, in case $p \geq 5$.
  This proves part (1) and also part (2) of the lemma in case $p \geq
  5$. Suppose that the conditions of part (2) hold with $p = 2$ or
  $3$. Then $|V| \leq 32$ or $|V| \leq 27$ in the respective
  cases. But in these cases one can check the validity of the lemma by
  GAP \cite{GAP}.

  Let us now turn to the proof of part (3) of the lemma. Suppose that
  $V$ is a primitive $G$-module and that $G$ is a solvable group. We
  may also assume that the minimal base size for $G$ on $V$ is at
  least $3$. By Theorem \ref{Dolfi} this happens only if case (3) or
  (4) of Theorem \ref{Dolfi} holds. But in these cases one can check
  by GAP \cite{GAP} that there always exists a vector $v$ in $V$ with
  $|\bC_{G}(v)| \leq |G|^{1/2}$, except when $|V|=3^4$ and $G$ is a
  unique subgroup (up to conjugacy) in $GL(V)$ of order $1152$.
\end{proof}

\begin{remark} Notice that $G = \GL(2,2) \wr \mathrm{Sym}(3)$ and
  $\GL(2,3) \wr \mathrm{Sym}(2)$ are imprimitive linear groups acting
  on a vector space $V$ of size $64$ and $81$ respectively with the
  property that there does not exist a vector $v \in V$ with
  $|\bC_{G}(v)| \leq |G|^{1/2}$. Therefore the second and third
  statements of Lemma~\ref{lemma key} are best possible in some sense.
\end{remark}


\section{Proofs of Theorems \ref{theorem-extendMoretoWolf} and
  \ref{primitive}}\label{Section 4}

Let $G$ be a finite $\pi$-solvable group where $\pi$ is the set of the
prime divisors of $|{\bF}^{\ast}(G/\Phi(G))|$. Then
${\bF}^{\ast}(G/\Phi(G)) = \bF(G/\Phi(G))$ and so $\bF(G/\Phi(G)) =
\bF(G)/\Phi(G)$ is a faithful $G/ \bF(G)$-module. By Gasch\"{u}tz's
theorem (see \cite[III.~4.5]{Huppert}) we also have that
$\bF(G)/\Phi(G)$ is a completely reducible $G/ \bF(G)$-module.
Therefore $\Irr(\bF(G)/\Phi(G))$ is a completely reducible and
faithful $G/ \bF(G)$-module as well. By applying Lemma~\ref{lemma
  key}, we have an irreducible character $\lambda$ of $\bF(G)/\Phi(G)$
(and thus of $\bF(G)$) whose stabilizer in $G/\bF(G)$ has size at most
$|G:\bF(G)|^{a}$ where $a = 1/2$ if the conditions of the second or
third statements of Lemma~\ref{lemma key} are satisfied and $a=2/3$
otherwise. This means that $\lambda$ lies in a $G$-orbit of size at
least $|G:\bF(G)|^{1-a}$.  Theorem~\ref{theorem-extendMoretoWolf} now
follows by Clifford's theorem.

Now we turn to the proof of Theorem~\ref{primitive}. Let $G$ be a
finite solvable group with $G/\bF(G)$ acting faithfully and
primitively on $\bF(G)/\Phi(G)$. Then $G/\bF(G)$ also acts faithfully
and primitively on $\Irr(\bF(G)/\Phi(G))$.  By the previous paragraph
we arrive to a conclusion unless the $G/\bF(G)$-module
$\Irr(\bF(G)/\Phi(G))$ has size $3^4$ and $|G/\bF(G)| = 1152$.  This
happens only if the faithful, primitive $G/\bF(G)$-module $V =
\bF(G)/\Phi(G)$ has size $3^4$.

As mentioned in the introduction, for a finite group $T$ let $k(T)$
be the number of conjugacy classes of $T$. The group $T = G/\Phi(G)$
has the form $HV$, a split extension of $H$ and $V$, where $H$ is
isomorphic to $G/\bF(G)$ and it acts the same way on $V$ as
$G/\bF(G)$ does (apply \cite[III.~4.4]{Huppert} to the group $T$
with a trivial Frattini subgroup). Thus $k(G/\Phi(G)) = k(HV)$. But
by GAP \cite{GAP} we have $k(HV) \leq |V|$ which gives $k(G/\Phi(G))
\leq |\bF(G)/\Phi(G)|$. This and Nagao's result~\cite{Nagao} then
gives \[k(G) \leq k(G/\Phi(G)) k(\Phi(G)) \leq |\bF(G)|.\] But then
the inequality $|G|/k(G) \leq b(G)^2$ yields the desired conclusion.


\section{Extendible characters of large degree in simple
  groups}\label{sectionsimplegroups}

Our proof of Theorem~\ref{main theorem
productofcompositionfactors} relies on the classification of
finite simple groups. The specific consequence of the
classification is the following result, which may have other
applications.

\begin{theorem}\label{main theorem 2}
  Let $S$ be a finite non-abelian simple group and let $d(S)$ denote
  the largest degree of an irreducible character of $S$ that extends
  to $\Aut(S)$. Then we have $d(S)\geq |S|^{1/3}$.
\end{theorem}

The proof of Theorem~\ref{main theorem 2} is fairly straightforward
for sporadic simple groups and simple groups of Lie type. So most of
this section is devoted to the proof for the alternating groups. We
first recall some basic combinatorics connecting partitions, Young
diagrams, and representation theory of the alternating and symmetric
groups.

Let $n$ be a natural number. It is well-known that there are
bijective correspondences between the partitions of $n$, the Young
diagrams of cardinality $n$, and the irreducible complex
characters of $\Sy_n$. Let $\lambda$ be a partition of $n$. That
is, $\lambda$ is a finite sequence $(\lambda_1,\lambda_2, \ldots
,\lambda_k)$ for some $k$ such that
$\lambda_1\geq\lambda_2\geq\cdots \geq\lambda_k$ and
$\lambda_1+\lambda_2+\cdots+\lambda_k=n$. Each $\lambda_i$ is
called a part of $\lambda$.

The Young diagram corresponding to $\lambda$, denoted by
$Y_\lambda$, is defined to be the finite subset of $\NN\times\NN$
such that \[(i,j)\in Y_\lambda \text{ if and only if } i\leq
\lambda_j.\] Two partitions of $n$ whose associated Young diagrams
transform into each other when reflected about the line $y=x$ are
called conjugate partitions. The partition conjugate to $\lambda$ is
denoted by $\overline{\lambda}$. If $\lambda=\overline{\lambda}$
then $Y_\lambda$ is symmetric and we say that $\lambda$ is
self-conjugate. For each node $(i,j)\in Y_\lambda$, we define the
so-called \emph{hook length} $h(i,j)$ to be the number of nodes that
are directly above it, directly to the right of it, or equal to it.
More precisely,
\[h(i,j):=1+\lambda_j+\overline{\lambda}_i-i-j.\]
We denote by $\chi_\lambda$ or $\chi_{Y_\lambda}$ the irreducible
character of $\Sy_n$ corresponding to $\lambda$ and $Y_\lambda$. The
degree of $\chi_\lambda$ is given by the \emph{hook-length formula}
of J.\,S.~Frame, G.\,B.~Robinson, and R.\,M.~Thrall,
see~\cite{Frame-Robinson-Thrall}:
\[f_\lambda:=\chi_{\lambda}(1)=\chi_{Y_\lambda}(1)=
\frac{n!}{\prod _{(i,j)\in Y_\lambda}h(i,j)}.\]

The irreducible characters of $\Al_n$ can be obtained by
restricting those of $\Sy_n$ to $\Al_n$. More explicitly,
$\chi_{\lambda}\downarrow_{\Al_n}=\chi_{\overline{\lambda}}\downarrow_{\Al_n}$
is irreducible if $\lambda$ is not self-conjugate. Otherwise,
$\chi_{\lambda}\downarrow_{\Al_n}$ splits into two different
irreducible characters of the same degree. In short, the degrees
of the irreducible characters of $\Al_n$ are labelled by
partitions of $n$ and are given by
$$ \widetilde{f}_\lambda= \left\{\begin
{array}{ll}
f_\lambda & \text{ if } \lambda\neq\overline{\lambda},\\
f_\lambda/2 & \text{ if } \lambda=\overline{\lambda}.
\end {array} \right.$$

Given a partition $\lambda$ of $n$, we define $A(\lambda)$ and
$R(\lambda)$ to be the sets of nodes that can be respectively
added or removed from $Y_\lambda$ to obtain another Young diagram
corresponding to a certain partition of $n+1$ or $n-1$
respectively. It is known (see~\cite[\S2]{Larsen-Malle-Tiep} for
instance) that
\[|A(\lambda)|<\sqrt{2n}+1 \text{ and } |R(\lambda)|<\sqrt{2n}.\]
The branching rule~\cite[\S9.2]{James} asserts that the
restriction of $\chi_\lambda$ to $\Sy_{n-1}$ is a sum of
irreducible characters $\chi_{Y_\lambda \backslash \{ (i,j) \}}$
as $(i,j)$ goes over all nodes in $R(\lambda)$. Also, by Frobenius
reciprocity, the induction of $\chi_\lambda$ to $\Sy_{n+1}$ is a
sum of irreducible characters $\chi_{Y_\lambda \cup \{(i,j)\}}$ as
$(i,j)$ goes over all nodes in $A(\lambda)$.

\begin{lemma}\label{lemma alternating groups} Theorem~\ref{main theorem
2} is true for the simple alternating groups.
\end{lemma}

\begin{proof}
The validity of the lemma can be checked by computer for all values
of $n$ in the range $5 \leq n \leq 30$. We will prove by induction
on $n\geq 30$ that $d(\Al_{n+1})\geq (n+1)^{1/3}d(\Al_n)$. This then
implies the inequality $d(\Al_n)\geq (n!/2)^{1/3}$ immediately.

Let $\theta$ be an irreducible character of $\Al_n$ with $n\geq 30$
such that $\theta$ is extendible to $\Sy_n$ and
$\theta(1)=d(\Al_n)$. Assume that $\theta$ extends to
$\chi\in\Irr(\Sy_n)$. Let $\lambda$ and $Y$ be respectively the
partition and Young diagram corresponding to $\chi$. Consider the
induction of $\chi$ to $\Sy_{n+1}$. We have
\begin{equation}\label{equation 1}
  \chi^{\Sy_{n+1}}(1)=(n+1)\chi(1)=(n+1)d(\Al_n)
\end{equation} and
\begin{equation}\label{equation 2}
  \chi^{\Sy_{n+1}}=\sum_{(i,j)\in A(\lambda)} \chi_{Y \cup\{(i,j)\}}
\end{equation} by the branching rule. There are two cases
arising.

\medskip

(1) None of the Young diagrams in $\{Y\cup \{(i,j)\}\mid (i,j)\in
A(\lambda)\}$ are symmetric. Then the irreducible characters
$\chi_{Y \cup \{(i,j)\}}$ restrict irreducibly to $\Al_{n+1}$ so
that
\[\chi_{Y \cup \{(i,j)\}}(1)\leq d(\Al_{n+1})\] for every $(i,j)\in
A(\lambda)$. We therefore deduce that
\[\chi^{\Sy_{n+1}}(1)\leq |A(\lambda)|d(\Al_{n+1}).\]
Since $|A(\lambda)|< \sqrt{2n}+1$ and
$\chi^{\Sy_{n+1}}(1)=(n+1)d(\Al_n)$, it follows that
\[(n+1)d(\Al_n)\leq \lceil \sqrt{2n}\rceil d(\Al_{n+1}),\] where
$\lceil x\rceil$ denotes the smallest integer not smaller than
$x$. Thus as $n \geq 30$ we conclude that
\[d(\Al_{n+1})\geq \frac{n+1}{\lceil \sqrt{2n}\rceil}d(\Al_n) \geq
{(n+1)}^{1/3} d(\Al_n),\] as desired.

\medskip

(2) There is a symmetric Young diagram in $\{Y\cup \{(i,j)\}\mid
(i,j)\in A(\lambda)\}$. Then there is exactly one such diagram and at
most $\lceil \sqrt{2n}-1\rceil$ non-symmetric diagrams in $\{Y\cup
\{(i,j)\}\mid (i,j)\in A(\lambda)\}$. Let $Y'$ be that symmetric Young
diagram and $\mu$ be the corresponding partition. By the branching
rule, we have
\[{\chi_{Y'}}\downarrow_{\Sy_n}=\sum_{(i,j)\in R(\mu)} \chi_{Y'\backslash \{ (i,j) \}}.\]
Now there are two subcases arising.

\medskip

(a) None of the Young diagrams of the form $Y'\backslash \{ (i,j)
\}$ where $(i,j)\in R(\mu)$ are symmetric. Then the characters
associated to these diagrams restrict irreducibly to $\Al_n$ and
thus their degrees are at most $d(\Al_n)$. As $|R(\mu)|
<\sqrt{2n+2}$, we deduce that
\[\chi_{Y'}(1)= \sum_{(i,j)\in R(\mu)} \chi_{Y'\backslash \{ (i,j)
  \}}(1)\leq \lceil\sqrt{2n+2}-1\rceil d(\Al_n).\]

Combining \eqref{equation 1}, \eqref{equation 2}, and the last
inequality, we have
\begin{align*} (n+1)d(\Al_n)&= \sum_{(i,j)\in A(\lambda)} \chi_{Y \cup
\{(i,j)\}}(1)\\
&=\chi_{Y'}(1)+ \sum_{(i,j)\in A(\lambda), Y \cup \{(i,j)\}\neq Y' }
\chi_{Y \cup \{(i,j)\}}(1)\\
&\leq \lceil\sqrt{2n+2}-1\rceil d(\Al_n)+
(|\Al(\lambda)|-1)d(\Al_{n+1})\\
&\leq \lceil\sqrt{2n+2}-1\rceil d(\Al_n)+\lceil \sqrt{2n}-1\rceil
d(\Al_{n+1}).
\end{align*}
Thus as $n \geq 30$ we deduce that
\[d(\Al_{n+1})\geq \frac{ n+2-\lceil \sqrt{2n+2}\rceil}{\lceil
  \sqrt{2n}-1\rceil} d(\Al_n) \geq {(n+1)}^{1/3} d(\Al_n),\] as
wanted.

\medskip

(b) There is a symmetric Young diagram in $\{Y'\backslash
\{(i,j)\}\mid (i,j)\in R(\mu)\}$. Then there is exactly one such
diagram and at most $\lceil \sqrt{2n+2}-2\rceil$ non-symmetric
diagrams in $\{Y'\backslash \{(i,j)\}\mid (i,j)\in R(\mu)\}$. Let
$Y''$ be that symmetric Young diagram and $\nu$ be the corresponding
partition. So we have two symmetric Young diagrams $Y'$ and $Y''$
and $Y''$ is obtained from $Y'$ by removing a node. Therefore, if
another node is removed from $Y''$ to get a Young diagram (of size
$n-1$), the resulting diagram can not be symmetric. Therefore, again
by the branching rule,
\[\chi_{Y''}(1)\leq \lceil \sqrt{2n}-1\rceil d(\Al_{n-1}).\]
It follows that
\[\chi_{Y'}(1)\leq \lceil \sqrt{2n}-1\rceil d(\Al_{n-1})+\lceil \sqrt{2n+2}-2\rceil d(\Al_n).\]
Therefore, as in subcase (a),
\begin{align*} (n+1)d(\Al_n)
&=\chi_{Y'}(1)+ \sum_{(i,j)\in A(\lambda), Y \cup \{(i,j)\}\neq Y' }
\chi_{Y \cup \{(i,j)\}}(1)\\
&\leq \lceil \sqrt{2n}-1\rceil d(\Al_{n-1})+\lceil
\sqrt{2n+2}-2\rceil d(\Al_n)+\lceil \sqrt{2n}-1\rceil d(\Al_{n+1}).
\end{align*}
Using the induction hypothesis that $d(\Al_{n-1})\leq
n^{-1/3}d(\Al_n)$ and noting $n\geq 30$, we obtain
\[d(\Al_{n+1})\geq \frac{ n+3-\lceil \sqrt{2n+2}\rceil-\lceil \sqrt{2n}-1\rceil n^{-1/3}}{\lceil \sqrt{2n}-1\rceil} d(\Al_n) \geq {(n+1)}^{1/3}
d(\Al_n),\] as wanted. The proof is complete.
\end{proof}

\begin{remark}
It can be shown by viewing Young diagrams of rectangular shapes that for any $\delta > 0$ there exists an $N > 0$ such that whenever $n > N$ we have $d(\Al_{n}) > {n!}^{(1/2) - \delta}$.
\end{remark}

\begin{proof}[Proof of Theorem~\ref{main theorem 2}] If $S$ is a
simple sporadic group or the Tits group, the proof is a routine
check from the Atlas~\cite{Atl1}. As the alternating groups have
been already handled in Lemma~\ref{lemma alternating groups}, we
can assume that $S$ is a simple group of Lie type in
characteristic $p$. We then realize that $S$ has the so-called
\emph{Steinberg character} $\St_S$ of degree $\St_S(1)=|S|_p$, the
$p$-part of the order of $S$, and $\St_S$ is extendible to
$\Aut(S)$ (see~\cite{Feit}). Now we can check the inequality
$|S|_p>|S|^{1/3}$ easily by consulting the list of simple groups
and their orders, see~\cite[p. xvi]{Atl1} for instance.
\end{proof}

\begin{remark} The bounding constant $1/3$ in Theorem~\ref{main theorem
2} is tight since the order of $\PSL_2(q)$ is
$q(q^2-1)/\gcd(2,q-1)$ and the largest degree of an irreducible
character of $\PSL_2(q)$ extendible to $\Aut(\PSL_2(q))$ is either
$q$ or $q+1$.
\end{remark}


\section{Bounding the product of the orders of
non-abelian composition  factors}\label{section-product-arbitrary-groups}

In this section we prove Theorem~\ref{main theorem
productofcompositionfactors}. To do that, we first apply the main
result of the previous section to obtain the following.

\begin{proposition}\label{proposition solvable radical} Let $N$ be a
non-abelian minimal normal subgroup of a finite group $G$. Then
\[b(G)^{3}\geq b(G/N)^{3} |N|.\]
\end{proposition}

\begin{proof} Since $N$ is a
non-abelian minimal normal subgroup of the finite group $G$, we know
that $N\cong S\times \cdots\times S$, a direct product of say $k$
copies of a non-abelian simple group $S$. By Theorem~\ref{main
theorem 2}, we know that there exists a non-principal character
$\alpha\in\Irr(S)$ such that $\alpha$ extends to $\Aut(S)$ and
$\alpha(1)\geq |S|^{1/3}$. In particular, the product character
$\psi:=\alpha \times \cdots \times \alpha$ is invariant under
$\Aut(N)=\Aut(S)\wr \Sy_k$. By~\cite{Mattarei}, the character $\psi$
extends to $\Aut(N)$. Since $N$ is embedded into $G/\bC_G(N)$ and
$G/\bC_G(N)$ embeds into $\Aut(N)$, the character $\psi$ extends to
$G/\bC_G(N)$ and therefore $\psi$ extends to a character $\chi$ of
$G$.

Applying Gallagher's Theorem~\cite[Corollary~6.17]{Isaacs1}, we
deduce that there is a bijection $\lambda\mapsto \lambda\chi$
between $\Irr(G/N)$ and the set of irreducible characters of $G$
lying above $\psi$. In particular, by taking $\lambda$ to be an
irreducible character of largest degree, we deduce that
$b(G/N)\chi(1)=b(G/N)\psi(1)$ is a character degree of $G$. It
follows that
\[b(G)\geq b(G/N)\psi(1).\]

Recall that $\psi:=\alpha \times \cdots \times \alpha$ and
$\alpha(1)\geq |S|^{1/3}$. Thus $\psi(1)\geq |N|^{1/3}$. It
follows by the above inequality that \[b(G)\geq b(G/N)|N|^{1/3},\]
and this completes the proof.
\end{proof}

Theorem~\ref{main theorem productofcompositionfactors} now is a
consequence of Proposition~\ref{proposition solvable radical}.

\begin{proof}[Proof of Theorem~\ref{main theorem
productofcompositionfactors}] In light of
Proposition~\ref{proposition solvable radical}, we see that if $N$
is a minimal normal subgroup of a finite group $X$, then
\[b(X)\geq b(X/N) \text{ if } N \text{ is abelian, and}\]
\[b(X)\geq b(X/N)|N|^{1/3} \text{ if } N \text{ is non-abelian}.\]

Let $1=M_0<M_1<\cdots<M_k=G$ be a chief series of $G$. That is,
$M_{i+1}/M_i$ is a minimal normal subgroup of $G/M_i$ for every
$i=0,1, \ldots ,k-1$. Arguing by induction on the number of chief
factors of $G$, we deduce that the product of the orders of all
non-abelian chief factors of $G$ is at most $b(G)^3$, which also
means that the product of the orders of all non-abelian
composition factors of $G$ is at most $b(G)^3$. The proof is
complete.
\end{proof}



\section{Distinguished subgroups and the largest character degree}\label{section-other-bounds}


We need two preliminary lemmas to prove
Theorem~\ref{main-theorem-Hallsubgroup}.

\begin{lemma}
\label{Brauer} Let $G$ be a finite $p$-solvable group and $V$ an
elementary abelian normal $p$-subgroup of $G$ for some prime $p$.
Suppose that $G/V$ acts faithfully on $V$. Then $k(G) \leq |G|_{p}$.
\end{lemma}

\begin{proof}
The hypothesis implies that $\bO_{p'}(G) = 1$. But then
Theorem~2.6(a) of~\cite{kgvbook} states that $G$ has a unique
$p$-block, the principal block. The defect group of the principal
block is a Sylow $p$-subgroup of $G$. Hence by the $p$-solvable case
of Brauer's $k(B)$-problem (see~\cite{kgv}) we obtain $k(G) \leq
|G|_{p}$.
\end{proof}

\begin{lemma}\label{Zoltan}
  Let $G$ be a $\pi$-solvable finite group, where $\pi$ is the set of
  prime divisors of $|\bF^*(G)|$. Then $\bF^*(G/\Phi(G))=
  \bF^*(G)/\Phi(G)$. In particular, the set of prime divisors of
  $|\bF^*(G/\Phi(G))|$ is a subset of $\pi$.
\end{lemma}
\begin{proof}
  Let $\overline{G} =G/\Phi(G)$ and let $N$ be the inverse image of
  $\bF^*(\overline{G})$ in $G$, that is, $N/\Phi(G)=\bF^*(G/\Phi(G))$.
  Then $\bF^*(\overline{G})=E(\overline{G})\circ
  \bF(\overline{G})$, the central product of $E(\overline{G})$, the
  subgroup generated by all the quasisimple subnormal subgroups of
  $\overline{G}$ and $\bF(\overline{G})$, the Fitting subgroup of
  $\overline{G}$. Moreover, $E(\overline{G})$ is itself a central
  product of subnormal quasisimple subgroups of $\overline{G}$. Here
  $\bZ(E(\overline{ G})) \leq \bF(\overline{G})$. As
  $\bF(\overline{G})=\bF(G)/\Phi(G)$, we have that $N/\bF(G)$ is a
  direct product of some non-abelian simple groups. (Unless,
  $N=\bF(G)$, in which case we are done.) Since $N/\bF(G)$ is a $\pi$-solvable
  group, we have that it is a $\pi'$-group. As $\bF(G)$ is a
  $\pi$-group, $(|\bF(G)|,|N:\bF(G)|)=1$, so the Schur-Zassenhaus
  theorem can be applied. We get $N=\bF(G)\rtimes K$ for some
  $\pi'$-group $K\leq G$.

  Now, we prove that we can assume that $\Phi(\bF(G))=1$. Assuming the
  converse, let $\tau:G\to G/\Phi(\bF(G))$ be the natural
  homomorphism. Then $\Phi(\bF(G))\leq \Phi(G)$, so
  $\Phi(\tau(G))=\tau(\Phi(G))$ holds. Furthermore,
  $\tau(\bF(G))=\bF(\tau(G))$, and $\bF^*(\tau(G))\leq \tau(N)$ also
  follows.

  We claim that $K$ acts faithfully on $\tau(\bF(G))$. Indeed, let
  $L\leq K$ be the kernel of this action and let $P$ be the Sylow
  $p$-subgroup of $\bF(G)$ for some $p\in \pi$. Then $L$ acts
  trivially on the vector space $P/\Phi(P)\simeq \tau(P)$. Now, we use
  the idea of the proof of \cite[III. Satz 3.17]{Huppert}, which we
  recall here only for completeness. Let $\{a_1,\ldots,a_l\}$ be a
  basis of $P/\Phi(P)$. Then the set
  $X=\{(a_1t_1,\ldots,a_lt_l)\,|\,t_1,\ldots,t_l\in\Phi(P)\}$ is fixed
  by $L$, and $L/C_L(P)$ acts semiregularly on $X$, since every
  element of $X$ provides a (minimal) generating set of $P$. It
  follows that $|L/C_L(P)|\mid |X|=|\Phi(P)|^l$. But $L$ is a
  $p'$-group, while $\Phi(P)$ is a $p$-group. So $L/C_L(P)=1$. As this
  holds for every Sylow subgroup of $\bF(G)$ we get $L\leq
  C_K(\bF(G))=1$.

  As $E(\tau(G))\leq \tau(K)$ centralizes $\bF(\tau(G))$, we get
  $E(\tau(G))=1$, so $\bF^*(\tau(G))=\bF(\tau(G))=\tau(\bF(G))$ is a
  $\pi$-group and we can use an induction argument on the order of the
  finite group.

  Now, let $\Phi(\bF(G))=1$. Then $\bF(G)$ is a direct product of
  elementary abelian groups, and it is completely reducible as a
  $K$-module by Maschke's theorem. As $\bF^*(\overline{G})=K\times
  \bF(\overline{G})$, we get $\bF(G)=\Phi(G)\oplus U$ as a $K$-module,
  where $U$ is centralized by $K$ and $K$ acts faithfully on $\Phi(G)$
  by conjugation.

  Let now $M= \bN_G(K)\geq UK$. We claim that $G=M\Phi(G)$. Notice
  that this is sufficient to prove the lemma, since $G=M$ implies that
  $K$ is normal in $G$ and so $K\leq \bF^*(G)$.

  To prove that $G=M\Phi(G)$, let $g\in G$ be arbitrary. Then $K^g\leq
  N\nor G$, and $K^g$ is also a complement of $\bF(G)$ in $N$. Thus,
  there is an $n\in N$ such that $K^g=K^n$ by the second part of the
  Schur-Zassenhaus theorem. So $g=m_1n$ for some $m_1\in M$. But
  $N=\Phi(G)UK$, where $UK\leq M$, so $N=(M\cap N)\Phi(G)$. Thus,
  $n=m_2t$ for some $m_2\in M,\ t\in \Phi(G)$. It follows that
  $g=m_1n=m_1m_2t\in M\Phi(G)$, and the proof is complete.
\end{proof}

\begin{proof}[Proof of Theorem \ref{main-theorem-Hallsubgroup}]
Suppose that we have the result for groups in consideration with a
trivial Frattini subgroup. Then for an arbitrary finite
$\pi$-solvable group $G$ where $\pi$ is the set of the prime
divisors of $|\bF^*(G)|$, we see that
$G/\Phi(G)$ is $\mu$-solvable where $\mu$ is the set of the prime
divisors of $|\bF^*(G/\Phi(G))|$, a subset of $\pi$ by Lemma \ref{Zoltan}. Therefore
\[k(G/\Phi(G)) \leq |G/\Phi(G)|_\mu \leq |G/\Phi(G)|_\pi=|G|_\pi/|\Phi(G)|,\] since every prime divisor of $|\Phi(G)|$ is in $\pi$.
Using Nagao~\cite{Nagao}, we then deduce that
\[k(G) \leq k(\Phi(G)) k(G/\Phi(G)) \leq |\Phi(G)|\cdot (|G|_\pi/|\Phi(G)|)=|G|_\pi,\] as wanted. Thus we may assume that $\Phi(G) = 1$.

Notice that the hypothesis on $G$ implies that ${\bF}^{*}(G) =
\bF(G)$ and thus $\bF(G)$ is a direct product of elementary abelian
groups and $X := G/\bF(G)$ acts faithfully on $V := \bF(G)$.
Moreover by \cite[III. 4.4]{Huppert}, we may also assume that $G$ is
the split extension $XV$ of a subgroup $X$ of $G$ and $V$.

We now proceed by induction on the size $n$ of the set $\pi$.

Let $n=1$. Then $V$ is an elementary abelian $p$-group for some
prime $p$ and $X$ acts faithfully on $V$. Thus by Lemma~\ref{Brauer}
we have $k(G) \leq |G|_{p}$, as desired.

So suppose that $n > 1$ and that the theorem is true for $n-1$. Let
$V = W \oplus U$ where $W$ is the maximal elementary abelian
$p$-subgroup of $V$ for some prime divisor $p$ of $V$.

Put $Y := XU \cong G/W$ so that $G = YW$. Put $N := \bC_{Y}(W)$.
Then $N$ contains $U$ as a normal subgroup and $N$ is normal in $G$.
In particular, $\bF^{*}(N) = \bF(N)$ (since a component in $N$ is
also a component in $G$). Now notice that $N/U$ acts faithfully on $U$. From this it follows that
the sets of the prime divisors of $|\bF(N)|$ and $|U|$ are the same
and equal to $\mu:=\pi\backslash \{p\}$. As it is clear that $N$ is
$\mu$-solvable, the induction hypothesis implies that $k(N) \leq
|N|_{\mu}$, which in turn implies that $k(N)\leq |N|_\pi$.

Now $G/N \cong (Y/N)W$ and $Y/N$ acts faithfully on $W$. It follows
that $|\bF(G/N)|$ has exactly one prime divisor, which is $p$. As
$G$ is $\pi$-solvable, the quotient $G/N$ is $p$-solvable. Hence by
Lemma~\ref{Brauer}, we have $k(G/N) \leq |G/N|_{p}$. Thus
$k(G/N)\leq |G/N|_\pi$.

Using Nagao~\cite{Nagao} again, we deduce that
\[k(G)\leq k(N)k(G/N)\leq |N|_{\pi}\cdot |G/N|_{\pi}=|G|_\pi.\]
This completes the proof of the first part of the theorem. The
second part follows readily as
\[|G:H|=|G|/|G|_\pi\leq |G|/k(G)\leq
b(G)^2.\]

For the last part, we just note that if $G$ is solvable then the set
of the prime divisors of $|\bF^*(G)|$ is precisely
that of $|\bF(G)|$.
\end{proof}

\begin{remark} Notice that the inequality $k(G) \leq |\bF(G)|$ implies Gluck's conjecture for
the solvable group $G$. However this inequality does not always hold
by considering $G = AGL(2,3)$, as kindly pointed out to us by
Robinson~\cite{Robinson2}. \end{remark}


As mentioned in the introduction, Theorem~\ref{main-theorem-b(G)^4}
is a consequence of results on the so-called \emph{commuting
probability} of finite groups. The commuting probability of a finite
group $G$, denoted by $\cp(G)$, is the probability that a randomly
chosen pair of elements of $G$ commute. That is,
\[\cp(G):=\frac{1}{|G|^2}|\{(x,y)\in G\times G\mid xy=yx\}|.\]

It is well-known that $\cp(G)=k(G)/|G|$ where $k(G)$ denotes the
number of conjugacy classes in $G$.

Using the classification of finite simple groups and early work of
R.~Kn\"{o}rr \cite{Knorr} on the coprime $k(GV)$-problem, Guralnick
and Robinson~\cite{Guralnick-Robinson} obtained a quite strong bound
for the commuting probability, namely $\cp(G)\leq
|G:\bF(G)|^{-1/2}$.

\begin{proof}[Proof of Theorem~\ref{main-theorem-b(G)^4}]
The result \cite[Theorem~10]{Guralnick-Robinson} implies that
\[\frac{k(G)}{|G|} \leq |G:\bF(G)|^{-1/2},\] which in turn implies
that
\[|G:\bF(G)|^{1/2} \leq \frac{|G|}{k(G)}.\]
Since the right-hand side of the previous inequality is no greater
than $b(G)^{2}$, we deduce that
\[|G:\bF(G)|^{1/2} \leq b(G)^2,\] and the theorem follows.\end{proof}

\begin{remark} Since the proof of the inequality $\cp(G)\leq
|G:\bF(G)|^{-1/2}$ by Guralnick and Robinson depends on the
classification of finite simple groups, the above proof of
Theorem~\ref{main-theorem-b(G)^4} depends on the classification as
well.
\end{remark}

\begin{proof}[Proof of Theorem \ref{IsaacsPassman}]
Using the result \cite[Theorem~D]{Isaacs-Passman} of Isaacs and
Passman, we have that the nilpotent group $\bF(G)$ has a subnormal
abelian subgroup, say $A$, of index at most $b(\bF(G))^4$. Combining
this with Theorem~\ref{main-theorem-b(G)^4}, we have
\[|G:A|=|G:\bF(G)|\cdot |\bF(G):A|\leq b(G)^4b(\bF(G))^4\leq
b(G)^8,\] which proves the first part of the theorem.

The result~\cite[Theorem 3]{LiebeckPyber} of Liebeck and Pyber is
that every finite group $G$ contains a solvable subgroup $S$ with
$k(G) \leq |S|$. This implies that
\[|G:S| \leq |G|/k(G) \leq {b(G)}^{2}\] and we have proved the second part of the theorem.
\end{proof}

We remark here that it is asked in \cite[page~539]{LiebeckPyber}
whether in fact any finite group $G$ contains a nilpotent subgroup
$N$ with $k(G) \leq |N|$. It would be interesting if one could
answer the following weaker question, which seems nontrivial to us.

\begin{question} Is it true that every finite group $G$ has a
nilpotent subgroup of index at most $b(G)^2$?
\end{question}


\section*{Acknowledgement} The authors are grateful to David Gluck for
his helpful comments and suggestions on an earlier version of the
manuscript. H.N. would like to dedicate this paper to Professor
Nguy\~\ecircumflex n~H.\,V.~H\uhorn ng for many years of friendship,
help, and encouragement.


\end{document}